\documentclass[11pt]{article}
\usepackage{amsmath}
\usepackage{amsfonts}
\usepackage{amssymb}
\usepackage{amsthm}
\usepackage{color}
\usepackage[utf8]{inputenc}
\theoremstyle{plain}
\newtheorem{Theorem}{Theorem}[section]
\newtheorem{Lemma}[Theorem]{Lemma}
\newtheorem{Proposition}[Theorem]{Proposition}
\newtheorem{corollary}[Theorem]{Corollary}
\newtheorem{question}[Theorem]{Question}
\newtheorem*{definition}{Definition}
\newtheorem*{remark}{Remark}

\newcommand{\footremember}[2]{%
	\footnote{#2}
	\newcounter{#1}
	\setcounter{#1}{\value{footnote}}%
}
\newcommand{\footrecall}[1]{%
	\footnotemark[\value{#1}]%
} 

\title{Measures of maximal entropy of bounded density shifts}
\author{Felipe Garc\'ia-Ramos \footremember{UASLP}{Institute of Physics, Universidad Aut\'onoma de San Luis Potos\'i} \footremember{JU}{Institute of Mathematics, Uniwersytet Jagiello\'nski} \and Ronnie Pavlov \footremember{DU}{Department of Mathematics, University of Denver} \and Carlos Reyes \footrecall{UASLP}} 
\begin{document}
	\maketitle
	\begin{abstract}
		We find sufficient conditions for bounded density shifts to have a unique measure of maximal entropy. We also prove that every measure of maximal entropy of a bounded density shift is fully supported. As a consequence of this we obtain that bounded density shifts are surjunctive.
	\end{abstract}

	\section{Introduction}
	
	The concept of entropy is of particular interest when trying to define formally how a system behaves at equilibrium. Given a dynamical system, we say that an invariant measure is a uniform equilibrium state if it achieves the maximal possible entropy. 
	It has been of interest to physicist and mathematicians to determine whether a system has a unique equilibrium state or not. When this happens some mathematicians say the system is intrinsically ergodic and some physicist say the system does not have phase transition. 
	
	%In this paper we are interested in trying to determine if bounded density shifts are intrinsically ergodic. Bounded density shifts were introduced by Stanley in \cite{stanley}. These subshifts are defined similarly to $\beta$-shift but uses a function instead of a codification of a numerical $\beta$-expansion and a counting order to compare words instead of lexicographic. Bounded density shifts and $\beta$-shifts are hereditary \cite{kwietniak2011topological}. 
	In this paper we are interested in trying to determine if bounded density shifts are intrinsically ergodic. Bounded density shifts were introduced by Stanley in \cite{stanley}. These subshifts are defined somewhat similarly to the classical $\beta$-shifts in that they both are hereditary (\cite{kwietniak2011topological}), meaning that membership in the shift is preserved under coordinatewise reduction of letters. Whereas $\beta$-shifts are `bounded from above' by a specific sequence coming from a $\beta$-expansion, bounded density shifts are restricted by length-dependent bounds on the sums of letters in subwords.
	
	Stanley proved characterizations of when bounded density shifts are shifts of finite type, sofic, or specified which are remarkably similar to those proved in \cite{schmeling} for $\beta$-shifts.
	
	%When trying to characterize when a bounded density shift is a shift of finite type, sofic or specified, Stanley proved remarkably similar results as that of $\beta$-shifts.
	
	A very effective way of proving that a transitive $\beta$-shift is intrinsically ergodic is using the Climenhaga-Thompson decomposition \cite{especificacion} (see Section 2.3), which uses specification of a sub-language. Using this powerful result one can prove that $\beta$-shifts (and their factors) are intrinsically ergodic in a few lines (see \cite[Section 3.1]{especificacion}). 
	
	Proving that bounded density shifts are intrinsically ergodic seems more mysterious. In this paper we also use Climenhaga-Thompson's theorem theorem to prove a fairly general sufficient condition (Theorem~\ref{thm:main}). The main application of the result is the following. 
	\begin{Theorem}
		Let $X_f\subset \{0,...,m\}^{\mathbb{Z}}$ be a bounded density shift and $\alpha_{f}$ its limiting gradient. If $\alpha_{f} > \sum_{i=1}^{ m } \frac{i}{i+1}$ then $X_f$ is intrinsically ergodic.   
	\end{Theorem}
	
	%and achieve success in a seemingly big family. In contrast with $\beta$-shifts, instead of a few lines it takes several pages to check that the decomposition holds. 
	
	%It is easy to find examples that satisfy our extra hypothesis (Corollary \ref{cor:simple}), and actually we do not know if every bounded density subshift satisfies it (Question \ref{question}). Actually we conjecture that the answer of this question is positive at least for binary subshifts and that every bounded density shift is intrinsically ergodic. 
	
	It is not difficult to find examples with this property, for a binary shift all we need is $\alpha_{f}>1/2$. Furthermore we do not know if any bounded density subshift fails to satisfy the conditions of Theorem~\ref{thm:main} (Question \ref{question}). %It is easy to find examples that satisfy our extra hypothesis (Corollary \ref{cor:simple}), and actually we do not know if every bounded density subshift satisfies it (Question \ref{question}). 
	We conjecture that the answer of this question is positive at least for binary subshifts and that every bounded density shift is intrinsically ergodic.
	
	This is not the first paper to study intrinsic ergodicity of bounded density shifts. This has been done in \cite{climenhaga2019one,pavlov2022subshifts}. Our hypotheses are also much simpler, and provide proofs of intrinsic ergodicity for new classes of bounded density shifts.
	
	%Our hypothesis are different to the ones in those papers, and we can provide the result for new examples. This is because our condition is tailor-made for bounded density shifts and  \cite{climenhaga2019one,pavlov2022subshifts} study conditions for general subshifts. 
	
	Furthermore, we prove that every measure of maximal entropy of a bounded density shift (with positive entropy) is fully supported. This property is sometimes known as entropy minimality because it is equivalent to having lower topological entropy on  every proper subshift. As a consequence of this we prove that synchronized bounded density shifts are always intrinsically ergodic, and we also obtain surjunctivity of bounded density shifts. In the last section of the paper we prove that these shifts possess universality properties.  \\
	%implies that every proper subshift has lower topological entropy. As a consequence of this we prove that synchronized bounded density shifts are always intrinsically ergodic, and we also obtain surjunctivity of bounded density shifts. 
	
	\emph{Acknowledgments}: {The authors would like to thank Dominik Kwietniak for rewarding conversations and insights, and the anonymous referee for their valuable comments. The first author was supported by the Excellence Initiative Strategic Program of the Jagiellonian University with grant number U1U/W16/NO/01.03. The second author gratefully acknowledges the support of a Simons Collaboration Grant.}
	
	\section{Definitions and preliminary results}
	%%% Hasta aquí los textos son exactamente los mismos. 
	\subsection{Subshifts}
	We devote this section to collect some basic definitions in symbolic dynamics. For a broader introduction to subshifts, languages and their properties, see \cite{Marcus}.  
	%% first commment of the referee
	Let $\mathcal{A}$ be a finite set of symbols. 
	We say that $w$ is a \emph{word} if there exists $n\in \mathbb{N}$ such that $w\in\mathcal{A}^n$ and we denote the \emph{length} of $w$ by $\vert w \vert$. Let $\varepsilon$ denote the empty word, i.e. the word with no symbols.

	A word $u$ is a \emph{subword} of $w$ if $u = w_kw_{k+1} \ldots w_{l}$ for some $1 \leq k \leq l \leq \vert w \vert$. For words $w^{(1)}, \ldots, w^{(n)}$, we use $w^{(1)} \ldots w^{(n)}$ to represent their concatenation.
	We say that a word $u$ is a \emph{prefix} of $w$ if $u = w_1 \ldots w_{k}$ for some $1 \leq k \leq \vert w \vert$ and a suffix if $u = w_k \ldots w_{\vert w \vert}$ for some $1\leq k \leq \vert w \vert$, denote by $\text{Suf}(w)$ and $\text{Pre}(w)$ the sets of nonempty suffixes and prefixes respectively for $w$. %%\rp{(NOTE: these notations are never used, I think we should delete these if that is the case.)}
	
	We endow $\mathcal{A}^\mathbb{Z}$ with the product topology. When describing a point $x\in\mathcal{A}^\mathbb{Z}$ as a sequence, we use a dot to indicate the central position as follows, $x = \ldots x_{-1}.x_{0}x_{1} \ldots$, where $x_i$ to represent the ith coordinate of $x$. {We represent intervals of integers with $[i,j]$, and $x_{[i,j]}=x_ix_{i+1}...x_j$.} 
	
	The shift map $\sigma : \mathcal{A}^\mathbb{Z} \rightarrow  \mathcal{A}^\mathbb{Z}$ is defined by $\sigma (x) = \ldots x_{-1}x_0.x_1x_2 \ldots$. We say that a set $X \subseteq \mathcal{A}^\mathbb{Z} $ is a \emph{subshift} if it is closed and invariant under $\sigma$. 
	
	For any subshift $X$, let $$\mathcal{L}_n(X)=\{ w\in \mathcal{A}^n:\exists x\in X\text{ and } i,j\in \mathbb{Z} \text{ s.t } x_{[i,j]} = w \}.$$
	%the set of words with length $n$ such that there exist $x \in X$ and $i,j\in \mathbb{Z}$ such that $x_{[i,j]} = w$. 
	We define $\mathcal{L}(X) = \bigcup_{i=0}^{\infty} \mathcal{L}_n(X)$ as the \emph{language} of the subshift $X$. Given a word $w$ and $k\in \mathbb{Z}$, we define its \emph{cylinder set} as $[w]_k = \lbrace x \in X : x_{[k,k+\vert w \vert -1]} = w \rbrace$. The cylinder sets form a basis of the topology of $\mathcal{A}^\mathbb{Z}$.  
	
	%There is a similar notation and theory if we consider one-way infinite sequences of $\mathcal{A}$ indexed by $\mathbb{N}$ instead of $\mathbb{Z}$. 
	
	\subsection{Specification properties}
	
	%Since there are several notions of specification in the literature, we use a variant of classical specification definition. They share the interpretation to approximate arbitrary segments of orbit with a single trajectory.
	A subshift $X$ is \emph{specified} if there exists $M \in \mathbb{N}$ such that for all $u,w \in \mathcal{L}(X)$, there is a $v \in \mathcal{L}_M(X)$ such that $uvw \in  \mathcal{L}(X)$. {Following \cite{especificacion}}, we also define specification for subsets of the language.

	Let $X$ be a subshift, $\mathcal{G} \subset \mathcal{L}(X)$ and $t \in \mathbb{N}_0$. We say that $\mathcal{G}$ has \emph{specification (with gap size $t$)} if for all $m \in \mathbb{N}$ and $w^{(1)}, \ldots , w^{(m)} \in \mathcal{G}$, there exists $v^{(1)}, \ldots , v^{(m-1)} \in \mathcal{L}_t(X)$ such that 
	$$w= w^{(1)}v^{(1)}w^{(2)}v^{(2)} \cdots v^{(m-1)}w^{(m)} \in \mathcal{L}(X).$$ Moreover, if the cylinder $[w]_0$ contains a periodic point of period exactly $\vert w \vert + t$, then we say that $\mathcal{G}$ has \emph{periodic specification.}  
	%Fix $t \in \mathbb{N}$. Either of the following conditions define a specification property on $\mathcal{G}$ with gap size $t$.
	
	%\begin{enumerate}
	
	%[(Per)] Condition (S) holds,  and in addition, the cylinder $[x]$ contains a periodic point of period exactly $\vert x \vert + t$.
	%\end{enumerate}
	
	\subsection{Measures of maximal entropy}
	
	For any subshift $X$, we denote by $M(X)$ the set of Borel probability measures on $X$. {Equipped with the weak* topology $M(X)$ is a compact topological space.} 
	
	For any $\mu \in M(X)$ and any finite measurable partition $\xi$ of $X$, %$\xi = \lbrace [w^{(1)}],  \ldots , [w^{(n)}] \rbrace$ of $X$, the 
	the \emph{entropy of $\xi$} (with respect to $\mu$), denoted by $H_\mu(\xi)$, is defined by
	
	\begin{equation*}
		H_\mu(\xi) = - \sum_{A \in \xi} \mu(A) \log \mu(A), %\mu \left( \left[ w^{(i)} \right] \right) \log \left( \mu \left( \left[ w^{(i)} \right] \right) \right).
	\end{equation*}
	where terms with $\mu(A) = 0$ are omitted.
	
	Given a subshift $X$ we denote the $\sigma$-invariant Borel probability measures with $M(X,\sigma)$.
	For $\mu \in M(X, \sigma)$, the \emph{entropy of $\mu$} (for the shift map $\sigma$) is defined by
	
	\begin{equation}
		h_\mu(X) = 
		\lim_{n \rightarrow \infty } \dfrac{-1}{n} \sum_{w \in \mathcal{L}_n(X)} \mu \left(  [w]_0 \right) \log \mu \left( [w]_0 \right) 
		= \lim_{n \rightarrow \infty } \dfrac{-1}{n} H_\mu(\xi^{(n)}),
	\end{equation}
	where $\xi^{(n)}$ represents the partition of $X$ into cylinder sets from the first $n$ letters, i.e. 
	$\xi^{(n)} = \{[w]_0 \ : \ w \in \mathcal{A}^n\}$.
	
	We note for future reference that $\xi^{(n)} = \bigvee_{i = 0}^{n-1} \sigma^{-i} \xi^{(1)}$, where $\xi^{(1)}$ is the partition based on $x_0$ and $\vee$ is the join of partitions. We will later need to make use of the following basic facts about entropy; for  proofs and general introduction to entropy theory, see \cite{Walters}.

	\begin{Theorem}[Theorem 4.3 \cite{Walters}]
		\label{theoreminequalityentropy}
		{For any subshift $X$, $\mu \in M(X)$, and $\mathcal{\xi}$, $\mathcal{\eta}$ finite partitions of $X$}, $H_\mu \left( \mathcal{\xi} \vee \mathcal{\eta} \right) \leq H_\mu\left( \xi \right) + H_\mu \left( \eta \right)$.
	\end{Theorem}
	
	\begin{Theorem}[Corollary 4.2.1 \cite{Walters}]
		\label{corolariocotaH}
		{For any subshift $X$ and $\mu \in M(X)$, if $\xi$ %= \lbrace [w^{(1)}],  \ldots , [w^{(k)}] \rbrace$ 
			is a finite measurable partition of $X$ with $k$ sets}, then $H_\mu(\xi) \leq \log(k)$, {with equality} only when 
		{$\mu(A) = k^{-1}$ for all $A \in \xi$.}%$\mu \left( \left[ w^{(i)} \right] \right) = 1/k$ for all $i$.
	\end{Theorem}
	
	\begin{Theorem}[\cite{Walters}, {p. 184}]\label{affine}
		{For any subshift $X$, finite measurable partition $\xi$ of $X$, measures $\mu_i \in M(X)$, and $p_i \geq 0$ ($1 \leq i \leq n$) with $\sum_{i = 1}^n p_i = 1$, 
			$H_{\sum_{i=1}^{n} p_i \mu_i }(\xi) \geq \sum_{i=1}^n p_i H_{\mu_i}(\xi)$.}
	\end{Theorem}
	
	By the well-known Variational Principle, the supremum of $h_\mu(X)$ over all $\mu \in M(X, \sigma)$ is the \emph{topological entropy} $h_{top}(X)$ of $X$. %value at the supremum is called the \emph{topological entropy of $X$}, $h_{top}(X)$.
	For any subshift $X$, we have that
	
	\begin{equation}
		h_{top}(X)= \lim_{n \rightarrow \infty} \frac{1}{n} \log \vert \mathcal{L}_n(X) \vert .
	\end{equation}
	
	For general topological dynamical systems, the supremum above may not be achieved. However,
	every subshift has at least one measure of maximal entropy, that is $\nu\in M(X,\sigma)$ achieving the supremum above,
	meaning that $h_{\nu}(X) = h_{top}(X)$ (e.g. see \cite[Remark (2), pg 192]{Walters}). 
	
	We say a subshift is \emph{intrinsically ergodic} if there is only one (probability) measure of maximal entropy. 
	
	%% Linea 169 del main2.tex dropbox
	Every specified subshift is intrinsically ergodic \cite{bowen1974some}. This result was generalized in several works, including \cite{especificacion} and \cite{pavlovspec}. Before stating the result we need some extra definitions.  
	
	Given a collection of words $\mathcal{D} \subseteq \mathcal{L}(X)$ and $n \geq 1$, we define $\mathcal{D}_n = \mathcal{D} \cap \mathcal{L}_n(X)$. 
	%be the set of words of length $n \in \mathcal{D}$.
	We denote the \emph{growth rate} of $\mathcal{D}$ by 
	
	\begin{equation}
		h(\mathcal{D}) = \limsup_{n \rightarrow \infty} \dfrac{1}{n} \log \vert \mathcal{D}_n \vert.
	\end{equation}
	
	Note that $h(\mathcal{L}(X))=h_{top}(X).$
	%In \cite{especificacion} they give sufficient conditions for a subshift to be intrinsically ergodic, through some conditions that satisfy the decomposition of the language. 
	
	Following \cite{especificacion}, we say that $\mathcal{L}(X)$ \emph{admits a decomposition} $\mathcal{C}^p\mathcal{G} \mathcal{C}^s$ for $\mathcal{C}^p, \mathcal{G}, \mathcal{C}^s \subset \mathcal{L}(X)$ if every $w \in \mathcal{L}(X)$ can be written as $uvw$ for some $u \in \mathcal{C}^p$, $v \in \mathcal{G}$, $w \in \mathcal{C}^s$. For such a decomposition, we define the collection of words $\mathcal{G}(M)$ for each $M \in \mathbb{N}$ by
	
	\begin{equation}
		\mathcal{G}(M) = \lbrace  uvw \, : \, u \in \mathcal{C}^p , v \in \mathcal{G}, w \in \mathcal{C}^s, \vert u \vert \leq M, \vert w \vert \leq M \rbrace .
	\end{equation}
	
	Recall that Per$(n)$ denotes the set of points with period at most $n$ under $\sigma$. 
	%The next result is due to Climenhaga-Thompson \cite{especificacion}[Theorem C] and apply to  $\mathbb{N}-$subshifts and $\mathbb{Z}-$subshifts. 
	
	\begin{Theorem}{(Climenhaga and Thompson}
		\cite{especificacion}) 
		\label{theo-clim-thomp}
		. Let $X$ be a subshift whose language $\mathcal{L}(X)$ admits a decomposition $\mathcal{L}(X) = \mathcal{C}^p \mathcal{G} \mathcal{C}^s$, and suppose that the following conditions are satisfied:
		
		\begin{enumerate}
			\item $\mathcal{G}$ has specification.
			\item $h(\mathcal{C}^p \cup \mathcal{C}^s) < h_{\mbox{top}}(X)$.
			\item For every $M \in \mathbb{N}$, there exists $\tau$ such that given $v \in  \mathcal{G}(M)$, there exists words $u,w$ with $\vert u \vert \leq \tau, \, \vert w \vert \leq \tau$ for which $uvw \in \mathcal{G}$. 
		\end{enumerate}
		Then $X$ is intrinsically ergodic.
		Furthermore, if $\mathcal{G}$ has periodic specification, then 
		
		\begin{equation}
			\mu_n = \dfrac{1}{\vert \mbox{Per}(n)\vert} \sum_{x \in \mbox{Per}(n)} \delta_x,
		\end{equation}
		converges to the measure of maximal entropy in the weak* topology. 
	\end{Theorem}
	
	\begin{remark}
		Using results from \cite{pacifico2022existence}, Climenhaga explained in a blog post \cite{climenhagablog} that condition 3 is actually not required to prove uniqueness of the measure of maximal entropy. However, this condition is not difficult to check for bounded density shifts with positive entropy (Lemma~\ref{lemalenghtwords}) and so we verify it regardless. 
	\end{remark}
	
	\subsection{Bounded density shifts}
	Bounded density shifts were introduced in \cite{stanley} (see also \cite[Chapter 3.4]{bruin2022topological}).
	Let $f : \mathbb{N}_0 \rightarrow \left[ 0 , \infty \right) $ be a function. We say $f$ is \emph{canonical} if 
	\begin{itemize}
		\item $f(0)=0$,
		\item $f(m+1)\geq f(m)$ for all $m\geq 0$, and
		\item $f(m+n)\leq f(m)+f(n)$ for all $n,m\in \mathbb{N}$.
	\end{itemize}
	
	The \emph{bounded density shift} associated to a canonical function, $f$, is defined as follows:
	
	\begin{equation}
		X_f = \left\lbrace x \in \left( \mathbb{N}_0 \right)^\mathbb{Z} : \forall p \in \mathbb{N} \mbox{ and } \forall i \in \mathbb{Z} \,\, \sum_{r=i}^{i+p-1} x_r \leq f(p) \right\rbrace .
	\end{equation}
	
	Note that $X_f$ is a subshift on the alphabet $\mathcal{A}=\{0,1,...,\lfloor f(1) \rfloor\}$.
	
	{Actually, bounded density shifts can be defined for any function $f : \mathbb{N}_0 \rightarrow \left[ 0 , \infty \right)$, but {it was shown in \cite{stanley} that every bounded density shift can be defined by some canonical $f$}.}
	
	\begin{definition}
		Let $X_f$ be a bounded density shift, the limit 
		
		\begin{equation}
			\lim_{n \rightarrow \infty} \dfrac{f(n)}{n}
		\end{equation}
		is called the \emph{limiting gradient} and is denoted by $\alpha_{f}$.
	\end{definition}

	The existence of the limit is given by Fekete's lemma and the definition of canonical function; furthermore, the limit is an infimum, and so $f(n) \geq \alpha_{f} n$ for all $n$.
	
	There exist bounded density shifts with $\alpha_{f}=0$ but they are fairly trivial systems where the upper density of non-zero coordinates is always 0.
	A bounded density shift has positive topological entropy if and only if $\alpha_{f}>0$ (see \cite[Theorem 12]{kwietniak2011topological}) if and only if it is coded (determined by a labeled irreducible graph with possibly countably
	many vertices)  (\cite[Theorem 3.1]{stanley}).

	As we mentioned in the previous section, the specification property guarantees intrinsic ergodicity. For bounded density shifts, $X_f$ is specified with specification constant M if and only if $0^M$ is intrinsically synchronizing (\cite[Theorem 5.1]{stanley}). Bounded density shifts with positive topological entropy without specification can easily be constructed (\cite{stanley}).
	
	%In \cite{stanley}, there are characterizations for bounded density shifts for properties like transitivity, topological mixing, finite type, sofic, specification, synchronized and coded. %The inclusions for mixing subshifts are:
	
	%\begin{equation*}
	%\mbox{ finite type } \subset \mbox{ sofic } \subset \mbox{ specified } \subset \mbox{ synchronized } \subset \mbox{ coded. }    
	%\end{equation*}
	
	As we mentioned in the previous section, the specification property guarantees intrinsic ergodicity. For bounded density shifts, $X_f$ is specified with specification constant $M$ if and only if $0^M$ is intrinsically synchronizing \cite[Theorem 5.1]{stanley}. There exist bounded density shifts with positive topological entropy without specification (\cite{stanley}).
	
	A subshift $X$ with alphabet $A=\{0,1,...,n\}$ is \emph{hereditary} if every time there is $x\in X$ and $y\in A^{\mathbb{Z}}$ with $y_i\leq x_i$ $\forall i\in \mathbb{Z}$, then $y\in X$. 
	It is not difficult to check that bounded density shifts are hereditary. 
	
	\section{Intrinsic ergodicity}
	In this section we fix a binary bounded density shift $X_f$. We define
	%\begin{equation*}
	%\label{definitionofG}
	$$
	\mathcal{G} =  \left\lbrace w \in \mathcal{L}(X_f) \,\, : \mbox{ if } u \in \mbox{Pre}(w) \cup \mbox{Suf}(w), \mbox{ then } \,\, \frac{1}{\vert u \vert} \sum_{i=1}^{\vert u \vert } u_i < \alpha_{f} \right\rbrace \text{, and}
	$$
	$$
	\mathcal{B} =  \mathcal{C}^p = \mathcal{C}^s  =  \left\lbrace v \in \mathcal{L}(X_f) \,\, : \,\, \frac{1}{\vert v \vert} \sum_{i=1}^{\vert v \vert} v_i \geq \alpha_{f} \right\rbrace \cup \{ \epsilon \}.
	%\end{equation*}
	$$
	where $\epsilon$ denotes the empty word. 
	\begin{Lemma}
		\label{LemaLenguaje}
		The language $\mathcal{L}(X_f)$ admits a decomposition $\mathcal{B} \mathcal{G} \mathcal{B}$.
	\end{Lemma}

	\begin{proof}
		Let $z \in \mathcal{L} \left( X_f \right)$. Define $u$ to be the prefix of $z$ in $\mathcal{B}$ of maximal length (which may be the empty word $\epsilon$), and denote its length by $M \geq 0$. Let $z'$ be the {maximal} proper subword of $z$ that does not overlap with $u$, i.e. $z'=z_{[M+1, \vert z \vert]}$. Similarly, define $w$ to be the suffix of $z'$ in $\mathcal{B}$ of maximal length (which may be the empty word $\epsilon$), and denote its length by $N \geq 0$.
		
		{We write $y = z_{[M+1,\vert z \vert -N]}$, and assume for a contradiction that $y \notin \mathcal{G}$}. Then by definition, there exists a word {$v \in \mbox{Pre}(y) \cup \mbox{Suf}(y)$} with
		\begin{equation*}
			\dfrac{1}{\vert v \vert} \sum_{i=1}^{\vert v \vert} v_i \geq \alpha_{f}.
		\end{equation*}
		{If $v \in \mbox{Pre}(y)$, then $uv$ would be a prefix of $z$ in $\mathcal{B}$ longer than $u$, contradicting minimality of $u$. Similarly, if $v \in \mbox{Suf}(y)$, then
			$vw$ would be a suffix of $z'$ in $\mathcal{B}$ longer than $w$, contradicting minimality of $w$.
			Therefore, we have a contradiction and $y \in \mathcal{G}$, and so $z = uyw \in \mathcal{B} \mathcal{G} \mathcal{B}$.}  
		%Without loss of generality we can assume that $v \in \mbox{Pre}\left( z_{[M+1,\vert z \vert -N-1]} \right)$, it means that $v \in \mathcal{B}$, but it is not possible because then $uv$ would be a prefix of $z$ in $\mathcal{B}$ longer than $u$. % is the prefix of $z$ with the largest size in $\mathcal{B}$. 
	\end{proof}

	%\newpage
	\begin{Lemma}
		\label{lemaespecifiacion}
		The set $\mathcal{G}$ has specification.
	\end{Lemma}
	
	\begin{proof}
		We will show that $\mathcal{G}$ has periodic specification with gap size $t=0$. Let $ m \in  \mathbb{N}$, $w^{(1)}, \ldots , w^{(m)} \in \mathcal{G}$, $v^{(a)}\in \mbox{Suf}(w^{(m)})$, $v^{(b)}\in \mbox{Pre}(w^{(1)})$ and $z = v^{(a)}w^{(1)} \cdots w^{(m)}v^{(b)}$. 
		We compute
		\begin{eqnarray*}
			\sum_{i=1}^{|z|} z_i & = & \sum_{i=1}^{|v^{(a)}|} v_{i}^{(a)} + \sum_{i=1}^{|w^{(1)}|} w_{i}^{(1)} + \ldots + \sum_{i=1}^{|w^{(m)}|} w_{i}^{(m)} + \sum_{i=1}^{|v^{(b)}|} v_{i}^{(b)} \\
			& < & |v^{(b)}| \alpha_{f} + |w^{(1)}| \alpha_{f} + \ldots + |w^{(m)}| \alpha_{f} +|v^{(b)}| \alpha_{f}\\
			& = & \alpha_{f} \left(|v^{(a)}| +\sum_{i=1}^{m} |w^{(i)}| +|v^{(b)}| \right) \\
			& = & \alpha_{f} |z| \\
			& \leq & f(|z|). 
		\end{eqnarray*}
		This implies that any periodic point made from concatenations of words from $\mathcal{G}$ is in $X_f$. We conclude that $\mathcal{G}$ has periodic specification.
	\end{proof}
	
	In the second part of the following proposition we use techniques from Misiurewicz's proof of the variational principle \cite{misiurewicz1976short} to build measures with entropy higher or equal than that of a sub-language. These applications of the tools from \cite{misiurewicz1976short} have already been noted in \cite[Proposition 5.1]{burns2018unique} and \cite[Lemma 6.8]{pacifico2022existence}.  
	
	\begin{Proposition}
		\label{lemaentropia}
		There exists $\mu\in M(X_f,\sigma)$ with $\sum_{i=0} ^{\lfloor f(1) \rfloor} i\mu([i]_0) \geq \alpha_{f}$ and $h(\mathcal{B}) \leq h_\mu(X_f)$.
	\end{Proposition}
	
	\begin{proof}
		%We will construct an invariant measure whose entropy is great or equal to the growth rate of the number of words of length $n$ in $B$. 
		For each $n \in \mathbb{N}$ and $w \in \mathcal{L}_n(X_f) \cap \mathcal{B}$, consider the set:
		
		\begin{equation*}
			K_n = \lbrace \left.^\infty 0 . w 0^\infty \right. : w \in \mathcal{L}_n(X_f) \cap \mathcal{B} \rbrace.
		\end{equation*}
		By construction $\vert K_n \vert = \vert \mathcal{L}_n(X_f) \cap \mathcal{B} \vert$. Let $\nu_n \in M(X_f)$ be the atomic measure concentrated uniformly on the points of $K_n$, i.e. 
		
		\begin{equation*}
			\nu_n = \dfrac{1}{\vert K_n \vert} \sum_{x \in K_n} \delta_x .
		\end{equation*}
		Let $\mu_n \in M(X_f)$ 
		be defined by
		\begin{equation*}
			\mu_n = \dfrac{1}{n} \sum_{j=0}^{n-1} \nu_n \circ \sigma^{-j}.
		\end{equation*}
		
		Note that
		
		\begin{eqnarray*}
			\sum_{i=0} ^{\lfloor f(1) \rfloor} i\mu_n([i]_0) & = & \sum_{i=0} ^{\lfloor f(1) \rfloor} \dfrac{i}{n} \sum_{j=0}^{n-1} \nu_n \circ \sigma^{-j} ([i]_0) \\
			& = & \sum_{i=0} ^{\lfloor f(1) \rfloor} \dfrac{i}{n} \sum_{j=1}^{n} \dfrac{\vert \{w\in \mathcal{L}_n(X_f) \cap \mathcal{B}: w_i=i \}\vert}{\vert K_n \vert} \\
			& = & \dfrac{1}{\vert K_n \vert} \sum_{w \in \mathcal{L}_n(X_f) \cap \mathcal{B}} \left( \dfrac{1}{n} \sum_{j=1}^{n} w_j \right) \\
			& \geq & \alpha_{f}.
		\end{eqnarray*}

		% \begin{eqnarray*}
			%  \mu_n([1]) & = & \dfrac{1}{n} \sum_{i=0}^{n-1} \nu_n \circ \sigma^{-i} ([1]). \\
			%            & = & \dfrac{1}{n} \left( \nu_n[1] + \nu_n(\sigma^{-1}([1])) + \ldots + \nu_n (\sigma^{n-1}([1])) \right) \\
			%           & = & \dfrac{1}{n} \left( \dfrac{1}{\vert K_n \vert} \sum_{x\in K_n} \delta_x ([1]) + \ldots + \dfrac{1}{\vert K_n \vert} \sum_{x\in K_n} \delta_x \left( \bigcup_{w \in \mathcal{B}_{n-1}(X_f) \cap B} [w1] \right) \right) \\
			%          & = & \dfrac{1}{n} \left( \dfrac{1}{\vert K_n \vert} \left( \sum_{x \in K_n} \left( \delta_x([1]) + \ldots + \delta_x \left( \bigcup_{w \in \mathcal{B}_{n-1}(X_f) \cap B} [w1] \right) \right) \right) \right) \\
			%         & \geq & \dfrac{1}{\vert K_n \vert} \sum_{w \in \mathcal{L}_n(X_f) \cap B} \left( \dfrac{1}{n} \sum_{i=1}^{n} w_i \right) \\
			%        & \geq & \alpha_{f}.
			%\end{eqnarray*}

			Since $M(X_f)$ is compact {(in the weak* topology)}, we can choose a subsequence such that 
			
			\begin{equation}
				\label{subsequenceofentropyofB}
				\lim_{j \rightarrow \infty} \dfrac{1}{n_j} \log \vert \mathcal{L}_{n_j} \left( X_f \right) \cap \mathcal{B} \vert   = \limsup_{n \rightarrow \infty} \dfrac{1}{n} \log \vert \mathcal{L}_n \left( X_f \right) \cap \mathcal{B} \vert = h(\mathcal{B}), 
			\end{equation}
			and {$\mu_{n_j}\rightarrow \mu \in M(X_f)$}. By the definition of $\mu_n$, it is routine to check that $\mu \in M \left( X_f, \sigma \right)$, i.e. $\mu$ is $\sigma$-invariant.

			We will use techniques from the proof of the variational principle in \cite{misiurewicz1976short} to prove that
			
			% s(\varepsilon, K, \sigma) = \limsup_{n \rightarrow \infty} (1/n) \log s_n(\varepsilon,k)
			
			\begin{equation}
				h_\mu  \left( X_f \right) \geq \limsup_{n \rightarrow \infty} \dfrac{1}{n} \log \vert \mathcal{L}_n (X_f) \cap \mathcal{B} \vert = h(\mathcal{B}).
			\end{equation}
			Firstly, since $\sum_{i=0} ^{\lfloor f(1) \rfloor} i\mu_{n_j}([i]_0) \geq \alpha_{f}$ {and $\mu_{n_j} \rightarrow \mu$}, we also have that $\sum_{i=0} ^{\lfloor f(1) \rfloor} i\mu([i]_0) \geq \alpha_{f}$. Consider the partition given by the alphabet $\xi = \lbrace [0]_0, \ldots , [\lfloor f(1) \rfloor]_0 \rbrace$. %Observe that $\xi$ is a generator. 
			{Since all $w \in \mathcal{L}_{n_j}(X_f) \cap \mathcal{B}$ have equal measure $\nu_{n_j}([w]_0) = |K_{n_j}|^{-1}$ and all other $w \in \mathcal{A}^n_j$ have $\nu_{n_j}([w]_0) = 0$, by Theorem \ref{corolariocotaH}},
			%% Empezar a cambiar desde aquí, usar la partición natural en lugar de la partición que tenemos.
			%    \begin{comment}
				%    Let $\varepsilon>0$. Consider $N \in \mathbb{N}$ such  that $1/2^N < \varepsilon$ and define the partition
				
				%    \begin{equation}
					%        \xi = \lbrace [w] : w \in \mathcal{L}_N(X_f) \rbrace.
					%    \end{equation}
				
				%   This partition satisfies the following two conditions for every $w \in \mathcal{L}_N(X_f)$:
				%  \begin{enumerate}
					%        \item $\mbox{diam}([w]) < \varepsilon$,
					%        \item $\mu (\partial [w]) = 0$.
					%    \end{enumerate}
				%    \end{comment}
			%Computing the entropy of $\nu_{n_j}$ with respect to the partition $\vee_{i=0}^{n_j-1}\sigma^{-i}\xi$ we obtain
			\begin{equation}
				H_{\nu_{n_j}}  \left( \bigvee_{i=0}^{n_j-1} \sigma^{-i} \xi \right)
				=  - \sum_{w \in \mathcal{L}_{n_j}\left( X_f \right) {\cap \mathcal{B}}} \nu_{n_j}([w]_0)\log \nu_{n_j}([w]_0) 
				{ = \log}|\mathcal{L}_{n_j}\left( X_f \right) \cap \mathcal{B}|. \label{valordeH}
			\end{equation}

			%						\nonumber \\ \nonumber
			%            & = & - \sum_{w \in \mathcal{L}_{n_j} \left( X_f \right)} \left( \dfrac{1}{\vert K_{n_j} \vert} \sum_{x \in K_{n_j}} \delta_x ([w])  \right) \log \left( \dfrac{1}{\vert K_{n_j} \vert} \sum_{x \in K_{n_j}} \delta_x ([w])  \right) \\ \nonumber
			%            & \geq & -\sum_{w \in \mathcal{L}_{n_j} \left( X_f \right) \cap \mathcal{B}} \left( \dfrac{1}{\vert K_n \vert} \sum_{x \in K_{n_j}} \delta_x ([w])  \right) \log \left( \dfrac{1}{\vert K_{n_j} \vert} \sum_{x \in K_{n_j}} \delta_x ([w])  \right) \\ 
			%            & = & \log \vert K_{n_j} \vert. \label{valordeH}
			%    \end{equation}
		%    The equality (\ref{valordeH}) holds because every $w \in \mathcal{L}_{n_j}(X_f) \cap \mathcal{B}$ has $\nu_{n_j}$-measure $\frac{1}{\vert K_{n_j} \vert}$, and the word $w$ occurs exactly once in elements of $K_{n_j}$.

		Let $q,n \in \mathbb{N}$ with $1 < q < n$ and define $a(t) = \lfloor  \frac{n-t}{q}\rfloor$ for $0 \leq t < q$. Note that $a(0) \geq a(1) \geq \cdots \geq  a(q-1)$. 
		{For every $0 \leq t \leq q-1$, we define} 
		\begin{equation*}
			S_t = \lbrace 0 , 1 , \ldots , t-1 , t + a(t)q, t+a(t)q+1, \ldots, n-1 \rbrace.
		\end{equation*}
		So, {for any such $t$}, we can rewrite $\lbrace 0, 1, \ldots, n-1 \rbrace$ as follows
		\begin{equation}
			\label{particionintervalo}
			\lbrace 0, 1, \ldots , n-1 \rbrace = \lbrace t + rq  + i \vert 0 \leq r < a(t), 0 \leq i < q \rbrace \cup S_t.
		\end{equation} 
		%%$S= \lbrace 0, 1 , \ldots , n-1 \rbrace \setminus \lbrace j + rq +i : 0 \leq r \leq a(j)-1, 0 \leq i \leq q-1 \rbrace$.
		
		%Note that $S= \lbrace 0, 1 , \ldots , n-1 \rbrace \setminus \lbrace j,j+1,..., j + a(j)q-1\rbrace.$
		
		%$S = \lbrace 0 , 1 , \ldots , j-1 , j + a(j)q, j+a(j)q+1, \ldots, n-1 \rbrace$.
		Observe that
		\begin{equation*}
			t + a(t)q = t + \left\lfloor \dfrac{n-t}{q}\right\rfloor q \geq t + \left( \dfrac{n-t}{q} - 1 \right) q = t+ n -t -q = n-q. 
		\end{equation*}
		Thus, the cardinality of $S_t$ is at most $2q$.

		Using \eqref{particionintervalo} we get
		
		\begin{equation}
			\label{particiondeljoin}
			\bigvee_{i=0}^{n_j-1} \sigma^{-i} \xi = \left( \bigvee_{r=0}^{a(t)-1} \sigma^{-(rq+t)} \bigvee_{i=0}^{q-1} \sigma^{-i} \xi \right) \vee \bigvee_{l \in S_t} \sigma^{-l} \xi.
		\end{equation}
		%%\rp{I THINK THE $\vee$ I COLORED BLUE SHOULD BE REMOVED, RIGHT?} A: yes
		Combining {\eqref{valordeH}}, \eqref{particiondeljoin} and Theorem \ref{theoreminequalityentropy} we obtain
		
		%(\ref{valordeH}), the fact $\vert  \mathcal{L}_n(X_f) \cap B \vert = \vert K_n \vert$, (\ref{particionintervalo}) and (\ref{cuentas_de_j})  we have that
		\begin{eqnarray}
			\log \vert \mathcal{L}_{n_j} \left( X_f \right) \cap \mathcal{B} \vert & {=} & H_{\nu_{n_j}} \left( \bigvee_{i=0}^{n_j-1} \sigma^{-i} \xi \right) \nonumber \\ \nonumber
			& \leq & \sum_{r=0}^{a(t)-1} H_{\nu_{n_j}} \left( \sigma^{-(rq+t)} \bigvee_{i=0}^{q-1} \sigma^{-i} \xi \right) + \sum_{l \in S_t} H_{\nu_{n_j}} \left( \sigma^{-l} \xi \right) \\
			& \leq & \sum_{r=0}^{a(t)-1} H_{\nu_{n_j} \circ \sigma^{-(rq+t)}} \left( \bigvee_{i=0}^{q-1} \sigma^{-i} \xi \right) + 2q \log (l). \label{otracotaparahtop}
		\end{eqnarray}
		
		For the inequality $\sum_{l \in S_t} H_{\nu_{n_j}}(\sigma^{-l}\xi) \leq 2q \log (l)$ we apply Theorem \ref{corolariocotaH}.
		{We note that} for each $0 \leq t \leq q-1$, we have
		
		\begin{equation}
			\left( a(t)-1 \right)q + t \leq \left\lfloor \dfrac{n-t}{q} - 1  \right\rfloor q + t = n - q.
		\end{equation}
		Summing the first term in the last line of (\ref{otracotaparahtop}) over $t$ from $0$ to $q-1$, and using that the numbers $\lbrace t+rq : 0 \leq t \leq q-1, 0 \leq r \leq a(t)-1 \rbrace$ are all distinct and are all no greater than $n-q$, yields
		
		%\begin{eqnarray*}
		%       q \log \vert \mathcal{L}_n \left( X_f \right) \cap B \vert & \leq & \sum_{j=0}^{q-1} \left( \sum_{r=0}^{a(j)-1} H_{\nu_n \circ \sigma^{-rq+j}} \left( \bigvee_{i=0}^{q-1} \sigma^{-i} \xi \right) \right) + 2q^2 \log (k) \\
		%      & = & \sum_{r=0}^{a(0)-1}H_{\nu_n \circ \sigma^{-(rq)}} \left( \bigvee_{i=0}^{q-1} \sigma^{-i} \xi \right) + \cdots + \sum_{r=0}^{a(q-1)-1} H_{\nu_n \circ \sigma^{-(rq+q-1)}} \left( \bigvee_{i=0}^{q-1} \sigma^{-i} \xi \right)\\
		%      & & \,\,\,\,\,\, + 2q^2 \log(k) \\
		%     & = & \sum_{p=0}^{n-1} H_{\nu_n \circ \sigma^{-p}} \left( \bigvee_{i=0}^{q-1} \sigma^{-i} \xi \right) + 2q^2 \log(k) .
		%\end{eqnarray*}
		
		\begin{eqnarray}
			& & \sum_{t=0}^{q-1} \left( \sum_{r=0}^{a(t)-1}  H_{\nu_{n_j}  \circ \sigma^{-(rq+t)}} \left( \bigvee_{i=0}^{q-1}\sigma^{-i} \xi \right) \right)   =  \sum_{r=0}^{a(0)-1}H_{\nu_{n_j} \circ \sigma^{-(rq)}} \left( \bigvee_{i=0}^{q-1} \sigma^{-i} \xi \right) + \cdots \nonumber \\
			& & \cdots + \sum_{r=0}^{a(q-1)-1} H_{\nu_{n_j} \circ \sigma^{-(rq+q-1)}} \left( \bigvee_{i=0}^{q-1} \sigma^{-i} \xi \right) \nonumber \\
			& & = \sum_{p=0}^{{n_j}-1} H_{\nu_{n_j} \circ \sigma^{-p}} \left( \bigvee_{i=0}^{q-1} \sigma^{-i} \xi \right) \label{valorfinalH}.
		\end{eqnarray}
		
		Using \eqref{otracotaparahtop} and \eqref{valorfinalH} we get
		
		\begin{equation*}
			q \log \vert \mathcal{L}_{n_j} (X_f) \cap \mathcal{B} \vert \leq \sum_{p=0}^{n_j-1} H_{\nu_{n_j} \circ \sigma^{-p}} \left( \bigvee_{i=0}^{q-1} \sigma^{-i} \xi \right) + \dfrac{2q^2}{n_j} \log (l).
		\end{equation*}
		
		{Now, we divide by $n_j$ and apply Theorem~\ref{affine} (with $p_i = \frac{1}{n_j}$), to obtain}
		
		\begin{equation}
			\dfrac{q}{n_j} \log \vert \mathcal{L}_{n_j} \left( X_f \right) \cap \mathcal{B} \vert \leq H_{\mu_{n_j}} \left( \bigvee_{i=0}^{q-1} \sigma^{-i} \xi \right) + \dfrac{2q^2}{n_j^2} \log (l).
			\label{desigualdadfinal}
		\end{equation}
		
		%We know that the boundaries of  $\bigvee_{i=0}^{q-1} \sigma^{-i} \xi  = \lbrace [w] : w \in \mathcal{L}_{N+q-1} \left( X_f \right) \rbrace$ have $\mu-$measure zero. Hence, for each $w \in \mathcal{L}_{N+q-1} \left( X_f \right)$ and for a subsequence $\lbrace \mu_{n_{j_k}} \rbrace$ of the subsequence in \eqref{subsequenceofentropyofB}
		
		%    \begin{equation*}
			%        \lim_{k \rightarrow \infty} \mu_{n_{j_k}} \left( [w] \right) = \mu \left( [w] \right).
			%    \end{equation*}
		We will also use that 
		\begin{equation}
			\lim_{k \rightarrow \infty} H_{\mu_{n_{j_k}}} \left( \bigvee_{i=0}^{q-1} \sigma^{-i} \xi \right) = H_\mu \left( \bigvee_{i=0}^{q-1} \sigma^{-i} \xi \right),
			\label{valordellimiteH}
		\end{equation}
		which is obtained using the definition of weak* convergence.
		Then, combining \eqref{desigualdadfinal} and \eqref{valordellimiteH} yields
		
		\begin{eqnarray*}
			qh(\mathcal{B}) & = & \lim_{k \rightarrow \infty} \dfrac{q}{n_{j_k}} \log \vert \mathcal{L}_{n_{j_k}} \left( X_f \right) \cap \mathcal{B} \vert \\
			& \leq & \lim_{k \rightarrow \infty} H_{\mu_{n_{j_k}}}   \left( \bigvee_{i=0}^{q-1} \sigma^{-i} \xi \right) + \lim_{k \rightarrow \infty} \dfrac{2q^2}{n_{j_k}} \log (l) \\
			& = & H_\mu \left( \bigvee_{i=0}^{q-1} \sigma^{-i} \xi \right).
		\end{eqnarray*}
		
		{Now, by definition of $h_{\mu}(X_f)$,}
		
		%Thus, since $\xi$ is a generator and by the Kolmogorov-Sinai Theorem. \cite[Theorem 4.17]{Walters}.
		\begin{equation*}
			h(B) \leq \lim_{q \rightarrow \infty} {\frac{1}{q}} H_\mu  \left( \bigvee_{i=0}^{q-1} \sigma^{-i} \xi \right) = h_\mu(X_f).
		\end{equation*}   
		%%  $q\cdot h(\mathcal{B}) \leq H_{\mu} \left( \bigvee_{i=0}^{q-1} \sigma^{-i} \xi \right)$. We conclude that $h(\mathcal{B}) \leq h_\mu(X_f)$.
	\end{proof}

	\begin{Lemma}
		\label{lemalenghtwords}
		For every $M \in \mathbb{N}$, there exists $\tau$ such that given $v \in \mathcal{G}(M)$, there exist words $u,w$ with $\vert u \vert \leq \tau$, $\vert w \vert \leq \tau$ for which $uvw \in \mathcal{G}$.
	\end{Lemma}
	
	\begin{proof}
		Let $M \in \mathbb{N}$ and $v \in \mathcal{G}(M)$. This implies that there exist $u', w' \in \mathcal{B}, v' \in \mathcal{G}$ such that $v = u' v' w'$ and $\vert u' \vert \leq M, \vert w' \vert \leq M$. Choose $u=w=0^\tau$, with $\tau =  \left\lceil \dfrac{2M \lfloor f(1) \rfloor}{\alpha_{f}}  \right\rceil$.
		
		%One can prove that that  $uvw \in \mathcal{G}$ by cases using that $u'$ and $w'$ are small (smaller than $M$) and that $v' \in \mathcal{G}$.  
		%For example,
		Let $z \in \mbox{Pre}(0^{ \tau}u'v'w'0^{ \tau})$. Consider the following sets, $N_1 = \left[ 1 , \tau \right] $, $N_2 = \left[ \tau + 1 , \tau + \vert u' \vert \right] \cup \left[ \tau + \vert u' v' \vert +1, \tau + \vert  u'v'w' \vert \right] $ and $N_3 = \left[ \tau + \vert u'  \vert +1, \tau + \vert u'v' \vert   \right]$. Note that $N_2$ corresponds to the section where $u'$ and $w'$ appear and $N_3$ where $v'$ appears. 
		Also, we can assume that $\vert z \vert \geq  \tau$ (otherwise we are considering that $z \in \text{Pre}(0^\tau)$), then
		
		%\begin{comment}
		%\begin{eqnarray*}
		%\frac{1}{\vert z \vert} \sum_{i=1}^{\vert z \vert } z_i & = &
		%    \dfrac{1}{ 2 \tau + \vert u'v'w' \vert } \left( \sum_{i=1}^{\vert u' \vert } u'_i + \sum_{i=1}^{\vert v' \vert } v'_i + \sum_{i=1}^{\vert w' \vert } w'_i \right) \\
		%    & \leq & \dfrac{2M + \vert v' \vert }{ \tau + \vert v \vert}\\
		%    & \leq & \dfrac{2M + \vert v' \vert }{ \left( \dfrac{2M+\vert v' \vert + \alpha_{f} \vert v \vert }{   \alpha_{f}} \right)} \\
		%    & = & \alpha_{f} \left( \dfrac{2M + \vert v' \vert}{2M+\vert v' \vert + \alpha_{f} \vert v \vert} \right) \\
		%   & \leq & \alpha_{f}.
		%\end{eqnarray*}
		%\end{comment}
		
		\begin{eqnarray*}
			& & \dfrac{1}{\vert z \vert} \sum_{i=1}^{\vert z \vert} z_i = \dfrac{1}{\vert z \vert } \left( \sum_{i \in N_1 \cap [1, \vert z \vert ] } z_i + \sum_{i \in N_2  \cap [1, \vert z \vert ] } z_i + \sum_{i \in N_3  \cap [1, \vert z \vert ] } z_i \right) \\
			& = & \dfrac{1}{\vert z \vert } \left( \frac{\vert N_1 \cap [1, \vert z \vert ] \vert}{\vert N_1 \cap [1, \vert z \vert ] \vert}  \sum_{i \in N_2  \cap [1, \vert z \vert ] } z_i \,\, + \,\, \frac{\vert N_3 \cap [1, \vert z \vert ] \vert}{\vert N_3 \cap [1, \vert z \vert ] \vert} \sum_{i \in N_3 \cap [ 1, \vert z \vert]} z_i \right)
			%\,\,\,\, \text{since } v' \in \mathcal{G} 
			\\
			& \leq & \dfrac{1}{\vert z \vert} \left( \frac{\vert N_1 \cap [1, \vert z \vert ] \vert}{\vert N_1 \cap [1, \vert z \vert ] \vert} 2M  \lfloor f(1) \rfloor  + \alpha_{f} \vert N_3 \cap [1, \vert z \vert ] \vert \right)  %\,\,\,\, \text{since }  \vert z \vert \geq \tau, \text{ then } \vert N_1 \cap [1,\vert z \vert ] \vert = \tau 
			\\
			& = & \dfrac{1}{\vert z \vert} \left( \vert N_1 \cap [1, \vert z \vert ] \vert  \frac{2M \lfloor f(1) \rfloor}{\tau}  + \alpha_{f} \vert N_3 \cap [1, \vert z \vert ] \vert \right)  %\,\,\,\, \text{since } \tau \geq \dfrac{2M \lfloor f(1) \rfloor }{\alpha_{f}}  
			\\
			& \leq & \dfrac{1}{\vert z \vert} \left( \alpha_{f}  \vert N_1 \cap [1, \vert z \vert ] \vert + \alpha_{f} \vert N_3 \cap [1, \vert z \vert ] \vert  \right) \\
			& = & \alpha_{f} \left( \dfrac{\vert N_1 \cap [1, \vert z \vert ] \vert + \vert N_3 \cap [1, \vert z \vert ] \vert }{\vert z \vert } \right)\\
			& \leq & \alpha_{f} 
		\end{eqnarray*}		
		Here, the first inequality holds since $v' \in \mathcal{G}$, the second equality holds because 
		$\vert N_1 \cap [1,\vert z \vert ] \vert = \tau$ (using $\vert z \vert \geq \tau$), and the second inequality holds since $\tau \geq \dfrac{2M \lfloor f(1) \rfloor }{\alpha_{f}}$.
		
		The proof for $z  \in \mbox{Suf}(0^{ \tau}u'v'w'0^{ \tau})$ is similar. 
	\end{proof}

	\begin{Theorem}
		\label{thm:main}
		Let $X_f$ be a bounded density shift. If every measure of maximal entropy $\mu$ has the property that $\sum_{i} ^{\lfloor f(1) \rfloor} i\mu([i]_0) < \alpha_{f}$, then $X_f$ is intrinsically ergodic, and
		
		\begin{equation}
			\mu_n = \dfrac{1}{\vert \mbox{Per}(n)\vert} \sum_{x \in \mbox{Per}(n)} \delta_x
		\end{equation}
		converges to the measure of maximal entropy in the weak* topology. 
	\end{Theorem}
	
	\begin{proof}
		If $\alpha_{f}=0$, then since all sequences have frequency $0$ of non-$0$ symbols, the unique invariant measure is the delta measure of 
		$\left.^\infty 0^{\infty} \right.$. 
		
		If $\alpha_{f}>0$ we will obtain the result using Theorem \ref{theo-clim-thomp}. First note that $\mathcal{B}= \mathcal{C}^p = \mathcal{C}^s$. Using Lemma \ref{LemaLenguaje} we obtain $\mathcal{L}(X) = \mathcal{C}^p \mathcal{G} \mathcal{C}^s$.  Now we will check the numbered hypotheses of Theorem \ref{theo-clim-thomp}.
		
		\begin{enumerate}
			\item Lemma \ref{lemaespecifiacion} gives us that $\mathcal{G}$ has specification.
			\item Let $\mu'$ be the measure constructed in Lemma \ref{lemaentropia}. By hypothesis it cannot be a measure of maximal entropy. Thus, $h(\mathcal{C}^p \cup \mathcal{C}^s) = h(\mathcal{B}) \leq h_{\mu'}(X_f)< h_{\mbox{top}}(X_f)$.
			\item We obtain this property using Lemma \ref{lemalenghtwords}. 
		\end{enumerate}
	\end{proof}
	
	The main application of the previous result that we have is the following. 
	
	\begin{corollary}
		\label{cor:simple}
		Let $X_f$ be a bounded density shift. If $\alpha_{f} > \sum_{i=1}^{\lfloor f(1) \rfloor} \frac{i}{i+1}$ then  $\sum_{i} ^{\lfloor f(1) \rfloor} i\mu([i]_0) < \alpha_{f}$ for every measure of maximal entropy $\mu$. This implies that $X_f$ is intrinsically ergodic, and
		
		\begin{equation}
			\mu_n = \dfrac{1}{\vert \mbox{Per}(n)\vert} \sum_{x \in \mbox{Per}(n)} \delta_x
		\end{equation}
		converges to the measure of maximal entropy in the weak* topology. 
	\end{corollary}
	
	\begin{proof}
		Using \cite[Corollary 4.6]{garcia2019extender} and the fact that bounded density shifts are hereditary we have that for any measure of maximal entropy $$\mu( [i]_0 )\leq \mu([i-1]_0).$$ Since $\mu$ is a probability measure this implies that $\mu([i]_0)\leq 1/(i+1)$. Thus, 
		$$
		\sum_{i=1}^{\lfloor f(1) \rfloor} i \cdot \mu ([i]_0) \leq \sum_{i=1}^{\lfloor f(1) \rfloor} \frac{i}{i+1}.
		$$
		We obtain the result using Theorem \ref{thm:main}.
	\end{proof}
	
	\begin{remark}
		In particular, every binary bounded density shift with $\alpha_{f} > 1/2$ is intrinsically ergodic. 
	\end{remark}
	
	Furthermore, we suspect that the hypothesis of Theorem \ref{thm:main} may always be satisfied, at least for binary subshifts, leading to the following questions.
	\begin{question}
		\label{question}
		Let $X$ be a hereditary binary subshift with positive topological entropy. Is it true that for any measure of maximal entropy $\mu$ we have that $\mu([1]_0)<\sup_{\nu\in M(X)} \nu ([1]_0)$?
	\end{question}
	A reason to suspect Question~\ref{question} is true is that if $X$ is hereditary and $\mu([1]_0)$ achieves its (positive) supremum, then it should be possible to increase the entropy of $\mu$ by allowing a small proportion of randomly chosen $1$ symbols to change to $0$s. Some circumstantial evidence is given by the class of $\mathcal{B}$-free shifts, for which it is known that maximal entropy is achieved by such a procedure	(cf. Theorem 2.1.8 of \cite{Bfree}). We also ask the corresponding question for bounded density shifts on larger alphabets. 
	
	\begin{question}
		Is it true that for every bounded density shift we have that $$\sum_{i} ^{\lfloor f(1) \rfloor} i\mu([i]_0) < \alpha_{f}$$ for every measure of maximal entropy?
	\end{question}
	
	%	a measure of maximal entropy should not give maximum concentration to a symbol. For example, in the binary fullshift the maximal entropy is not supported in the cylinder $[1]_0$.	(MAYBE A REFERENCE FOR B-FREE SUBSHIFTS?)
	
	%	\textcolor{blue}{I would delete this sentence; it seems redundant with the exposition before Question~\ref{question}.} A positive answer to Question~\ref{question} would imply that for every measure of maximal entropy $\mu$ on a binary bounded density shift we have that $\mu([1]_0) < \alpha_{f} $.

	One more natural question is whether we can prove stronger properties on the unique measure of maximal entropy via arguments such as those in \cite{burns2018unique} and \cite{pavlov2022subshifts}. 
	
	\begin{question}
		Let $X_f$ be an intrinsically ergodic bounded density shift. Does the measure of maximal entropy have the $K$-property? Is it Bernoulli?
	\end{question}
	
	%\textcolor{blue}
	{We don't know how to approach this question with current techniques. All arguments we're aware of which prove Bernoulli require connection to countable-state Markov shifts, which do not seem clear for bounded density shifts. And the usual argument to prove $K$-property (without Bernoulli) is to show that the product of $(X_f, \sigma)$ with itself has a unique measure of maximal entropy, but in general Climenhaga-Thompson decompositions are not preserved under products, and we do not see any reason that bounded density structure improves the situation. We note that purely being hereditary does not necessarily imply either property, as in \cite{Bfree} it was shown that for $\mathcal{B}$-free shifts, the unique measure of maximal entropy factors onto the so-called Mirsky measure, which is of zero entropy; this precludes the $K$-property.}
	
	%	One way to obtain the $K$-property is by checking that the product system is intrinsically ergodic (as noted in \cite{ledrappier1977mesures,pavlovspec,call2022k}). With our current techniques it is not clear that we can obtain this. 

	\section{Entropy minimality and surjunctivity}
	
	%We will use the next theorem to prove our result. 
	%\begin[\cite[]{garcia2019extender}]{Theorem}
	
	%\end{Theorem}
	We will now prove a property called entropy minimality for all bounded density shifts for $\alpha_{f} > 0$ using results from \cite{garcia2019extender}. We first need some definitions.
	
	A subshift $X$ is \emph{entropy minimal} if every subshift strictly contained in $X$ has lower topological entropy. Equivalently, $X$ is entropy minimal if every measure of maximal entropy on $X$ is fully supported.
	%Recall that $\lbrace \left.^\infty0.0^\infty \right \rbrace$ is the trivial bounded density shift. In the following statements we assume that $X_f$ is non trivial.

	%\begin{Lemma}
	%\label{lemaentropiapositiva}
	%Let $X_f$ be a transitive bounded density shift. Then $h_{\mbox{top}}(X_f) > 0$.\end{Lemma}

%\begin{proof}
%Using transitivity and compactness, we can find a point $y \in X_f$ such that the set $1(y) = \lbrace n \in \mathbb{N} \, \vert \,  y_n = 1 \rbrace$ satisfies
%\begin{equation*}
%\limsup_{n \rightarrow \infty}{\dfrac{\vert 1(y) \cap \lbrace  1, \ldots , n\rbrace \vert}{n}}  > 0.
%\end{equation*}
%Then we can find $\varepsilon>0$ and a sqeuence $v^{(k)}$ of words that occurs in $y$ such that $\vert v^{(k)} \vert \rightarrow \infty$ with $k  \rightarrow \infty$ and $\sum_{i=1}^{\vert w^{(k)} \vert} w^{(k)}_i \geq \vert w^{(k)} \vert \varepsilon$. Using \cite[Lemma 3]{kwietniak2011topological} we have that $\vert w^{(k)} \vert \varepsilon \leq \log \vert \mathcal{L}_{\vert w^{(k)} \vert} (X) \vert$ for all $k > 0$. It yields

%\begin{equation*}
%   h_{\mbox{top}}(X_f) = \lim_{n \rightarrow \infty} \dfrac{1}{n} \log \vert \mathcal{L}_n (X_f) \vert = \lim_{k \rightarrow \infty} \dfrac{1}{\vert w^{(k)} \vert } \log \vert \mathcal{L}_{\vert w^{(k)} \vert} (X_f) \vert \geq \varepsilon. 
%\end{equation*}

%\end{proof}

%To prove the result we will use extender sets.

Let $X$ be a subshift and $v \in \mathcal{L}(X)$. The  \emph{extender set} of $v$ is defined by
\begin{equation*}
	E_{X_f}(v) = \lbrace y \in \lbrace 0 , 1, \ldots, \lfloor f(1) \rfloor \rbrace^\mathbb{Z} : y_{(-\infty,0]}vy_{[1, \infty )} \in X_f \rbrace.
\end{equation*}

\begin{Theorem}[Garc\'ia-Ramos and Pavlov \cite{garcia2019extender}]
	\label{GRP}
	
	Let $X$ be a subshift with $h_{top}(X)>0$, $\mu$ a measure of maximal entropy and $v,w\in \mathcal{L}(X)$. If $E_X(v)\subseteq E_X(w)$ then 
	$$
	\mu(v)\leq \mu(w)e^{h_{top}(X)(|w|-|v|)}.
	$$
\end{Theorem}

\begin{Theorem}
	\label{entropyminimal}
	Every bounded density shift (with $\alpha_{f}>0$) is entropy minimal.
\end{Theorem}

\begin{proof}
	Let $X_f$ be a bounded density shift, $\mu \in M(X_f, \sigma)$ a measure of maximal entropy and $w \in \mathcal{L}(X_f)$. Since the topological entropy of $X_f$ is positive then $1 \in \mathcal{L}(X_f)$, and $\mu([1]_0)>0$ (otherwise $\mu([0]_0)=1$ and the entropy cannot be positive). By  Poincar\'e's recurrence theorem, there exists $v'\in \mathcal{L}(X_f)$ for which $\mu([v']_0) > 0$ and
	\[
	\sum_{i=1}^{|v'|} v'_i > \sum_{i=1}^{|w|} w_i.
	\]
	We can then define $v$ which is coordinatewise less than or equal to $w$ with 
	$$\sum_{i=1}^{|v|} v_i=\sum_{i=1}^{|w|} w_i.
	$$
	By the fact that $X_f$ is hereditary, $E_{X_f}(v') \subset E_{X_f}(v)$, and so by Theorem~\ref{GRP}, $\mu([v])\geq \mu([v'])>0$.
	
	%\subseteq [1]$, $\mu([v]) = \mu([1])>0$ such that, for each $x \in [v]$ there exists a sequence $n_1 < n_2 < \ldots$ of natural numbers with $\sigma^{n_i}(x) \in [v]$. 
	
	We want to prove that $E_{X_f}(v) \subseteq E_{X_f}(0^{\vert v \vert} w 0^{\vert v \vert})$. Let $y \in E_{X_f}(v)$, with $x = y_{(-\infty, 0]}.vy_{[1,\infty)}\in X_f$, and $x' = y_{(-\infty, 0]}.0^{\vert v \vert }w0^{\vert v \vert}y_{[1,\infty)}$.
	Let $n<m\in \mathbb{Z}$. We consider two cases, when $x'_{[n,m]}$ is a subword of $0^{\vert v \vert} w 0^{\vert v \vert}$ and when it is not. If $x'_{[n,m]}$ is subword of $0^{\vert v \vert} w 0^{\vert v \vert}$, then $x'_{[n,m]} \in \mathcal{L}(X_f)$ since $w \in \mathcal{L}(X_f)$ (\cite[Lemma 2.3]{stanley}).  
	%%%If $x'_{[n,m]}$ is a subword of $0^{\vert v \vert}w0^{\vert v \vert}$, then $x'_{[n,m]} \in \mathcal{L}(X_f)$ because $w \in \mathcal{L}(X_f)$ and $f$ is nondecreasing. 
	Otherwise, there exists $p\in \mathbb{Z}$ such that 
	
	$$
	\sum_{i=n}^{m} x'_i \leq \sum_{i=n+p}^{m+p} x_i \leq f(m-n). 
	$$
	%%\rp{(NOTE: I DON'T UNDERSTAND THE EXISTENCE OF $p$, AND I DON'T UNDERSTAND WHY THIS IS ENOUGH TO PROVE $x'_{[n,m]} \in \mathcal{L}(X_f)$; DON'T WE HAVE TO CONSIDER ALL SUBWORDS TO PROVE IT'S IN $X_f$?)}

	This implies that $x'_{[n,m]} \in \mathcal{L}(X_f)$. Thus, $x' \in X_f$, and so $y \in E_{X_f}(0^{\vert v \vert}w0^{\vert v \vert})$. Since $y$ was arbitrary, $E_{X_f}(v) \subseteq E_{X_f}(0^{\vert v \vert}w0^{\vert v \vert})$. Using Theorem \ref{GRP} we conclude that 
	$$
	\mu([w]_0)\geq \mu([0^{\vert v \vert}w0^{\vert v \vert}]_0) \geq \mu([v]_0)e^{-{h_{top}(X)(|w|-|v|)}}>0.
	$$
	Therefore, $\mu$ is fully supported.

\end{proof}

%A bounded density shift has positive topological entropy if and only if $\alpha_{f}>0$ (see \cite[Theorem 12]{kwietniak2011topological}) if and only if it is coded (determined by a labeled irreducible graph with possibly countably
%many vertices)  \cite[Theorem 3.1]{stanley}. 
%%DEFINIR SYNCHRONIZED

Let $X$ be a subshift. A word $v \in \mathcal{L}(X)$ is \emph{intrinsically synchronizing} if $uv,vw \in \mathcal{L}(X)$ then $uvw \in \mathcal{L}(X)$.

A subshift is \emph{synchronized} if there exists $v \in \mathcal{L}(X)$ such that $v$ is an intrinsically synchronizing word.

Every entropy minimal synchronized subshift is intrinsically ergodic \cite{thomsen2006ergodic,garcia2019extender} and every synchronized subshift is coded \cite{fiebig1992covers}. Hence, we obtain the following corollary. 
\begin{corollary}
	Every synchronized bounded density shift is intrinsically ergodic. 
\end{corollary}
%TODO BOUNDED DENSITY ES ENTROPY MINIMAL
%ES DECIR CADA MEDIDA DE ENTROPIA MAXIMA TIENE SOPORTE COMPLETO.  a m

%\subsection{Surjunctivity}

Another application of entropy minimality is surjunctivity. 
Given a subshift $X$, we say $\phi:X\rightarrow X$ is a \emph{shift-endomorphism} if it's continuous and it commutes with the shift. If a shift-endomorphism is bijective we say it is a \emph{shift-automorphism}. 

A subshift $X$ is said to be \emph{surjunctive} if every injective shift-endomorphism of $X$ is a shift-automorphism. Every full shift is surjunctive (\cite[Chapter 3]{coornaert2010cellular}. The following result is known (e.g. see \cite{ceccherini2021expansive}) but it is not explicitly stated. We write the proof since the argument is simple.  
\begin{Lemma}
	Every entropy minimal subshift is surjunctive. 
\end{Lemma}
\begin{proof}
	Let $X$ be a subshift and $\phi:X\rightarrow X$ an injective shift-endomorphism. This implies that $\phi(X)$ is a subshift which is topologically conjugate to $X$. Since topological entropy is conjugacy-invariant, $\phi(X)$ has the same topological entropy as $X$. If $X$ is entropy minimal then $\phi(X)=X$. 
\end{proof}

Using this and Theorem \ref{entropyminimal} we obtain the following. 
\begin{corollary}
	Every bounded density shift with positive topological entropy is surjunctive. 
\end{corollary}

\section{Universality}

{A dynamical system} is said to be universal if every system with smaller entropy can be embedded in the original system {(this can be studied either in the topological or measure-theoretic category)}. {For instance, measure-theoretic} universality of the full shift follows from Krieger's generator theorem. {Results about both types (topological and measure-theoretical) of universality have been proved} for systems with specification in \cite{quas2016ergodic,burguet2020topological,chandgotia2021borel}, and {we can prove a topological universality result for bounded density subshifts as well}. We first need some basic definitions about topological dynamical systems. 

A \emph{topological dynamical system} is a pair $(X,T)$ where $X$ is a compact metrizable space and $T:X\to X$ is a continuous function. 
Let $(X,T)$ and $(X',T')$ be two topological dynamical systems. We say $X$ and $X'$ are \emph{conjugated} if there exists a homeomorphism $f: X\to X'$ such that $T'\circ f=f\circ T$.

%\textcolor{blue}{I personally would strongly consider not defining topological entropy in general. It's just a black box for these results anyway (i.e. the reader loses nothing if they don't know the definition), and it opens the door for a really unpleasant conversation with the referee; if we needed general topological entropy for this paper, why didn't we define it in the definitions section? For the record, if the referee made this complaint, I'd agree with them. So one possible fix is just to say `entropy can be defined for general TDS, but we do not need the exact definition here.}

{ For any TDS $(X,T)$ one can assign a \emph{topological entropy} $h_{top}(X,T)$. When the system is a subshift the notion coincides with the definition in Section 2.3. For the definition see \cite[Chapter 7]{Walters}.}
%	\[h_{top}(X,T,\mathcal{U})=\lim_{n\to\infty}\frac{1}{{n}}\log N(\mathcal{U}^{{n}}),\]

Let $\gamma\in \mathbb{R}_+$. We say a subshift $X$ is $\gamma$\emph{-universal} if for any TDS with $h_{top}(X_1,T_1)< \gamma$ there
is a subshift $X'\subset X$ such that $(X_1,T_1)$ is conjugated to $(X',\sigma).$
\begin{Theorem}
	[Burguet \cite{burguet2020topological}]
	Every subshift $X$ with specification is $h_{top}({X})$-universal. 
	%Furthermore if for every $n \in \mathbb{N}$ and  $w \in \mathcal{L}_n(X)$ there exist $m \ge n$ and  $ w' \in C_m$ such that $ w' \mid_{[-n,n]} = w$ then $X$ is fully ergodic universal.
\end{Theorem}

%Let $X$ be a subshift and $g:\mathbb{N} \to \mathbb{N}$ a function so that $\lim_{n\to \infty}\frac{g(n)}{n}=0$.

%We say that  $\mathcal{C} = ( C_n)_{n=1}^\infty$ with $ C_n \subset \mathcal{L}_n(X)$ is a  \emph{flexible sequence of patterns} for $X$ (with respect to $g=0$) if  for every  $k<n \in \mathbb{N}$ any $k$-spaced subset $K$ contained in $[k-n,...,n-k]$ and $W \in ( C_k)^K$ there exists $w \in C_n$ so that 
%\begin{equation}\sigma^{i}(w)\mid_{[-k,k]}=W(i) \mbox{ for all } i \in K.\end{equation}
%We say that $\mathcal{C}$ as above is a \emph{flexible marker sequence of patterns} if in addition\begin{equation}\forall x \in X \mbox{ and } n \in\mathbb{N}, \mbox{ the set } \left\{i \in \mathbb{Z}:\sigma^{i}(x)\mid_{[-n,n]}\in  C_n \right\}\mbox{ is }  n\mbox{-spaced}. \end{equation}

Let $\alpha \in \mathbb{R}_+$. We define $X_{\alpha}$ as the bounded density shift obtained with the function $f(n)=\lfloor n\alpha \rfloor$. Using \cite[Theorem 1.3]{stanley} we have that $X_\alpha$ has specification. 
%The shift $X_{\alpha}$ is the smallest bounded density shift with limiting gradient $\alpha_f$.  

Given a bounded density shift $X_f$, one can check that %$\mathcal{G}=\mathcal{L}(X_{\alpha_f})$ 
$X_{ \alpha _f} \subset X_f$. 
%We have that \alpha_f \leq \frac{f(p)}{p}$ for every $p\in \mathbb{N}$. 
%\textcolor{blue}{This is quite false; you need to consider all subwords, not just prefixes. For instance, if $\alpha = 0.5$, then $0011$ satisfies your definition, but is clearly not in the language of $X_\alpha$. But do we even need this characterization??? I would say that the argument is self-contained in the following and you should just delete the sentence I'm complaining about.} 
Let $x \in X_{\alpha_f}$, then for every $i \in \mathbb{Z}$ and for every $p \in \mathbb{N}$ we have 

\begin{equation*}
	\sum_{r=i}^{i+p-1} x_r \leq \left\lfloor p \alpha_f \right\rfloor \leq p \alpha_f \leq p \dfrac{f(p)}{p} \leq f(p).
\end{equation*}
Therefore $x \in X_f$ and $X_{\alpha_f} \subset X_f$.

\begin{corollary}
	Let $X_f$ be a bounded density shift. We have that $X_f$ is $h_{top}(X_{\alpha_f})$-universal. 
\end{corollary}

\bibliographystyle{plain}
\bibliography{bibliography}
\end{document}